\documentclass[12pt]{amsart}
\usepackage{color}
\usepackage{amsfonts,amsmath,amsthm,amssymb,latexsym}
\usepackage[hidelinks]{hyperref}
\usepackage{tikz}
\usepackage{cancel}
\usepackage{natbib}

\usepackage{algorithm}
\usepackage{algorithmic}

\usepackage{mathrsfs}

\DeclareMathOperator{\ord}{ord}
\DeclareMathOperator{\Sol}{Sol(\qq)}
\DeclareMathOperator{\Param}{IFP(\qq)}
\DeclareMathOperator{\Pla}{Places}

\DeclareMathOperator{\RT}{RTrunc_N(\qq)}
\DeclareMathOperator{\ST}{STrunc_N(\qq)}
\newcommand{\C}{\mathbb{C}}

\newcommand{\cc}{\mathscr{C}(F)}

\newcommand{\cP}{\mathcal{P}}
\newcommand{\qq}{\mathbf{p}_0}

\newcommand{\im}{\mathrm{Im}}
\newcommand{\Places}{\Pla(\qq)}
\newcommand\para{\vspace*{2mm}}

\newtheorem{theorem}{Theorem}[section]
\newtheorem{proposition}[theorem]{Proposition}
\newtheorem{lemma}[theorem]{Lemma}
\newtheorem{corollary}[theorem]{Corollary}
\theoremstyle{definition}
\newtheorem{definition}[theorem]{Definition}
\newtheorem{example}[theorem]{Example}

\theoremstyle{remark}

\numberwithin{equation}{section}
\theoremstyle{remark}

\begin{document}

\date{\today}
\title[Existence and convergence of series solutions of ODE's]{Existence and convergence of Puiseux series solutions for autonomous first order differential equations}

\author{Jos\'{e} Cano}
\address{Dpto. Algebra, an\'alisis matem\'atico, geometr\'{\i}a y topolog\'{\i}a, Universidad de Valladolid, Spain.}
\email{jcano@agt.uva.es}

\author{Sebastian Falkensteiner}
\address{Research Institute for Symbolic Computation (RISC), Johannes Kepler University Linz, Austria.}
\email{falkensteiner@risc.jku.at}

\author{J.Rafael Sendra}
\address{Dpto. de F\'{\i}sica y Matem\'aticas, Universidad de Alcal\'a, Madrid, Spain.}
\email{rafael.sendra@uah.es}

\thanks{The first author was partially supported by MTM2016-77642-C2-1-P (AEI/FEDER, UE). The second and third authors were partially supported by FEDER/Ministerio de Ciencia, Innovaci\'{o}n y Universidades Agencia Estatal de Investigaci\'{o}n/MTM2017-88796-P (Symbolic Computation: new challenges in Algebra and Geometry together with its applications).
	The second author was also supported by the Austrian Science Fund (FWF): P 31327-N32. 
	The third author is member of the Research Group ASYNACS (Ref. CT-CE2019/683).}

\begin{abstract}
Given an autonomous first order algebraic ordinary differential equation $F(y,y')=0$, we prove that every formal Puiseux series solution of $F(y,y')=0$, expanded around any finite point or at infinity, is convergent.
The proof is constructive and we provide an algorithm to describe all such Puiseux series solutions.
Moreover, we show that for any point in the complex plane there exists a solution of the differential equation which defines an analytic curve passing through this point.
\end{abstract}

\maketitle

\noindent \textbf{keywords}
Algebraic differential equation,
algebraic curve,
place,
formal Puiseux series solution,
convergent solution.

\section{Introduction}
We study local solutions of, not necessarily linear, autonomous first order ordinary differential equations of the form $F(y,y')=0$, where $F(y,p)$ is a polynomial (or indeed a holomorphic function) in two variables.
Rational and algebraic solutions of these equations have been studied in \cite{feng2004rational,feng2006polynomial} and \cite{aroca2005algebraic}.
In particular, they found degree bounds of the possible rational or algebraic solutions such that these global solutions can be computed algorithmically.
In \cite{FalkensteinerSendra_2018} it is proven that any formal power series solution of an autonomous first order algebraic ordinary differential equations is convergent.
We extend this result to the case of fractional power series solutions and give an algorithm to compute all of them.

The problem of finding power series solutions of ordinary differential equations has been extensively studied in the literature.
A method to compute generalized formal power series solutions, i.e. power series with real exponents, and describe their properties is the Newton polygon method.
A description of this method is given in \cite{fine1889,fine1890} and more recently in \cite{GrSi:1991,DoraJung1997,Aroca2000}.
In \cite{Cano2005}, the first author, using the Newton polygon method, gives a theoretical description of all generalized formal power series solution of a non-autonomous first order ordinary differential equation as a finite set of one parameter families of generalized formal power series.
This description of the solutions is in general not algorithmic by several reasons.
One of them is that there is no bound on the number of terms which have to be computed in order to guarantee the existence of a generalized formal power series solution when extending a given truncation of a determined potential solution.
Also the uniqueness of the extension can not be ensured a-priori. 
The direct application of the Newton polygon method to autonomous first order differential equations does not provide any advantage with respect to the non-autonomous case, because during the computations the characteristic of being autonomous gets lost.

In \cite{RISC5589} they derive an associated differential system to find rational general solutions of non-autonomous first order differential equations by considering rational parametrizations of the implicitly defined curve.
We instead consider its places and obtain an associated differential equation of first order and first degree which can be treated by the Newton polygon method, described in \cite{CanoJ-1993A}.
Using the known bounds for computing places of algebraic curves (see e.g. \cite{Duval1989}), existence and uniqueness of the solutions and the termination of our computations can be ensured.

The structure of the paper is as follows. Section \ref{sec:preliminaries} is devoted to recall the preliminary theory on formal Puiseux series and algebraic curves used throughout the paper.
In Section \ref{sec-PSP} we show that every non-constant formal Puiseux series solution defines a place of the associated curve.
We give a necessary condition on a place of the curve to contain in its equivalence class formal Puiseux series solutions of the original differential equation, and show the analyticity of them.
In the case where the solutions are expanded around a finite point, the necessary condition turns out to be sufficient as well.
As a byproduct, we obtain a new proof of the fact that there is an analytic solution curve of $F(y,y')=0$ passing through any given point in the plane.
This result is a consequence of Section 6.10 in \cite{ArocaF_2000}.
In Section \ref{sec-algorithms} algorithms for computing all Puiseux series solutions are presented and illustrated by examples.
Subsection \ref{subsec-zero} is devoted to solutions expanded around zero.
For proving the correctness of the algorithm, we give a precise bound on the number of terms such that the solutions are in bijection with the corresponding truncations.
In Subsection \ref{subsec-infinity} we consider solutions expanded at infinity.
Here we are able to compute for every solution a corresponding truncation, but in this case we are not able to guarantee the uniqueness of the extension.

\section{Puiseux series solutions and places}\label{sec:preliminaries}

In this section we introduce the notation, assumptions, and main notions that will be used throughout this paper.

\para

Let us consider the differential equation
\begin{equation}\label{eq-main}
F(y,y')=0,
\end{equation}
where $F \in \C[y,p]$ is non-constant in the variables $y$ and $p$. 
We will study existence and convergence of formal Puiseux series solutions of \eqref{eq-main}.
Formal Puiseux series can either be expanded around a finite point or at infinity.
In the first case, since equation \eqref{eq-main} is invariant under translation of the independent variable, without loss of generality we can assume that the formal Puiseux series is expanded around zero and it is of the form $\varphi(x)=\sum_{j\geq j_0}a_j\,x^{j/n}$, where $a_j\in \C,$ $n \in \mathbb{Z}_{>0}$ and $j_0\in \mathbb{Z}$.
In the case of infinity we can use the transformation $x=1/z$ obtaining the (non-autonomous) differential equation $F(y(z),-z^2y'(z))=0$.
In order to deal with both cases in a unified way, we will study equations of the type
\begin{equation}\label{eq-infinity}
F(y(x),x^hy'(x))=0,
\end{equation}
with $h \in \mathbb{Z} \setminus \{1\}$ and its formal Puiseux series solutions expanded around zero.
We note that for $h=0$ equation \eqref{eq-infinity} is equal to \eqref{eq-main} and for $h=2$ the case of formal Puiseux series solutions expanded at infinity is treated. 
In the sequel, we assume that $h$ is fixed.

We use the notations $\C[[x]]$ for the ring of formal power series, $\C((x))$ for its fraction field and $\C((x))^*=\bigcup_{n \geq 1} \C((x^{1/n}))$ for the field of formal Puiseux series expanded at zero.
We call the minimal natural number $n$ such that $\varphi(x)$ belongs to $\C((x^{1/n}))$ the \textit{ramification order} of $\varphi(x)$. 
Moreover, for $\varphi(x)=\sum_{j\geq j_0}a_j\,x^{j/n}$ with $a_{j_0} \neq 0$ we call $j_0/n \in \mathbb{Q}$ the order of $\varphi$, denoted by $\ord_x(\varphi(x))$, and set $\ord_x(\varphi(x))=\infty$ for $\varphi=0$.

\para

Associated to \eqref{eq-infinity} there is an affine algebraic curve $C(F) \subset \C^2$ defined by the zero set of $F(y,p)$ in $\C^2$.
We denote by $\cc$ the Zariski closure of $C(F)$ in $\C_{\infty}^2,$ where $\C_{\infty}=\C\cup \{\infty\}$.
In addition we assume throughout the paper that $F$ has no factor in $\C[y]$ or $\C[p]$.

Additionally, we may require that a formal Puiseux series solution $y(x)$ of \eqref{eq-infinity} fulfills the initial conditions $y(0)=y_0, (x^hy'(x))(0)=p_0$ for some fixed $\qq=(y_0,p_0) \in \C_{\infty}^2$.
In the case where $y(0)=\infty$, $\tilde{y}(x)=1/y(x)$ is a Puiseux series solution of a new first order differential equation of the same type, namely the equation given by the numerator of the rational function $F(1/y,-x^h p/y^2)$, and $\tilde{y}(0)\in \C$.
Therefore, in the sequel, we assume that $\qq\in \C\times \C_{\infty}$.

\para

Here we recall some classical terminology, see e.g. \cite{walker1950algebraic}.
A \textit{formal parametrization} centered at $\qq\in \cc$ is a pair of formal Puiseux series $A(t)\in \C((t))^2\setminus \C^2$ such that $A(0)=\qq$ and $F(A(t))=0$.
In the set of all formal parametrizations of $\cc$ we introduce the equivalence relation $\sim$ by defining $A(t) \sim B(t)$ if and only if there exists a formal power series $s(t)\in \C[[t]]$ of order one such that $A(s(t))=B(t)$.
A formal parametrization is said to be reducible if it is equivalent to another one in $\C((t^m))^2$ for some $m>1$.
Otherwise, it is called irreducible.
An equivalence class of an irreducible formal parametrization $(a(t),b(t))$ is called a \textit{place} of $\cc$ centered at the common center point $\qq$ and is denoted by $[(a(t),b(t))]$.
In every place there is, up to the substitution of $n$-th root of unities, exactly one formal parametrization of the type $(a_0+t^n,b(t))$ and we refer to them as classical Puiseux parametrizations.
We observe that $\ord_t(a(t)-y_0)$ and $\ord_t(b(t))$ are independent of the representative $(a(t),b(t))$ of a place of $\cc$ centered at $\qq$.

\section{Puiseux solution places}\label{sec-PSP}
Let us consider the sets $\Sol$ containing the non-constant formal Puiseux series solutions of equation \eqref{eq-infinity} with initial values $\qq$, $\Param$ containing all irreducible formal parametrizations of $\cc$ at $\qq$ and $\Places$ containing the places of $\cc$ centered at $\qq$.
Let us define the mapping $\Delta:\Sol \longrightarrow \Param$ as $$\Delta(y(x))=\left(y(t^{n}),t^{hn}\, \frac{d\, y}{d\, x}(t^{n})\right),$$ where $n$ is the ramification order of $y(x)$ and denote by $\delta: \Sol \longrightarrow \Places$ the map $\delta(y(x))=[\Delta(y(x))]$.
The map $\Delta$ is well defined because on the one hand, $\Delta(y(x))$ is a formal parametrization of $\cc$ centered at $\qq$ and on the other hand, by the definition of the ramification order, one deduces that $\Delta(y(x))$ is irreducible.

We remark that, since $\Delta$ is well defined, a necessary condition for $y(x) \in \Sol$ is that $\qq \in \cc$.

\begin{definition}
	A place $\cP \in \Places$ is a \textit{(Puiseux) solution place} of \eqref{eq-infinity} if there exists $y(x)\in \Sol$ such that $\delta(y(x))=\cP$.
	Moreover, we say that $y(x)$ is a \textit{generating Puiseux (series) solution} of the place $\cP$.
	An irreducible formal parametrization $A(t)\in\Param$ is called a \textit{solution parametrization} if $A\in \im(\Delta)$.
\end{definition}

Note that the above definition generalizes the notion of solution place in \cite{FalkensteinerSendra_2018} for formal power series solutions to Puiseux series solutions.

Now we give a characterization for an irreducible formal parametrization to be a solution parametrization.
Later we will show how to decide whether a given place contains a solution parametrization, i.e. whether it is a solution place.

\begin{lemma}\label{lemma:nec_cond_DE}
	Let $y(x)\in \Sol$  be of ramification order $n$, and let $(a(t),b(t))=\Delta(y(x))$. It holds that
	\begin{align}
	&a'(t)=n\,t^{n(1-h)-1}\,b(t). \label{eq:necessary_DE} \\
	&n(1-h)=\ord_t(a(t)-y_0)-\ord_{t}(b(t)). \label{eq:nec_eq_order_ramification}
	\end{align}
\end{lemma}
\begin{proof}
	Since $a(t)=y(t^{n})$ and $b(t)=t^{hn}\,y'(t^{n})$, by the chain rule
	\begin{align*}
	a'(t)= n\,t^{n-1}\,y'(t^{n})=n\,t^{n(1-h)-1}\,b(t).
	\end{align*}
	Equation \eqref{eq:nec_eq_order_ramification} is obtained by taking the order in $t$ on both sides of equation \eqref{eq:necessary_DE}.
\end{proof}

\begin{proposition}\label{Theorem:Characterization Sol Param}
	Let $(a(t),b(t))\in \Param$. Then $(a(t),b(t))$ is a solution parametrization if and only if there exists $n\in \mathbb{Z}_{>0}$ such that equation \eqref{eq:necessary_DE} holds.
	In this case, $n$ is the ramification order of $(a(t),b(t))$.
\end{proposition}
\begin{proof}
	The first implication follows from Lemma \ref{lemma:nec_cond_DE}.
	Let us now assume that \eqref{eq:necessary_DE} holds for an $n \in \mathbb{Z}_{>0}$ and write $a(t)=y_0+\sum_{j \geq k} a_j\,t^{j}$ with $k>0$, $a_k\neq 0$, and $b(t)=\sum_{j \geq k-n(1-h)} b_j\,t^{j}$.
	Let us consider $y(x)=y_0+\sum_{j \geq k} a_j\,x^{j/n}$.
	By assumption, $y'(x)=x^{-h}\,b(x^{1/n})$ and $$F(y(x),x^hy'(x))=F(a(x^{1/n}),b(x^{1/n}))=0.$$ Thus, $y(x) \in \Sol$.
	It remains to show that $n$ is the ramification order of $y(x)$.
	Otherwise, there exists a natural number $m\geq 2$, such that $m$ divides $n$ and if $a_i\neq 0$ then $m$ divides $i$.
	By assumption, we have that $a_{j+n(1-h)}\neq 0$ if and only if $b_{j}\neq 0$.
	Hence, if $b_j\neq 0$, then $m$ divides $j$.
	This implies that $(a(t),b(t))$ is reducible in contradiction to our assumption.
	Therefore, $n$ is the ramification order of $y(x)$ and $\Delta(y(x))=(a(t),b(t))$.
\end{proof}

\begin{lemma}\label{lemma:ram order solution place}
	All Puiseux series solutions in $\Sol$, generating the same solution place in $\Places$, have the same ramification order.
	We call this number the \textit{ramification order} of the solution place.
	As a consequence, the map $\Delta$ is injective.
\end{lemma}
\begin{proof}
	Let $y_1,y_2 \in \Sol$ be such that $\delta(y_1)=\delta(y_2)$.
	Let $n$ and $m$ be the ramification orders of $y_1$ and $y_2$, respectively.
	Then there exists an order one formal power series $s(t)$ such that $\Delta(y_1)(s(t))=\Delta(y_2)(t)$.
	Let us denote $\Delta(y_i)$ as $\Delta(y_i)=(a_i,b_i)$ with $i=1,2$.
	By equation \eqref{eq:necessary_DE}
	\[ a_2'(t)=m\,t^{m(1-h)-1}\,b_2(t)=m\,t^{m(1-h)-1}\,b_1(s(t)) \]
	and
	\[ a_2'(t)=(a_1(s(t)))'=a_1'(s(t))\,s'(t)=n\,s(t)^{n(1-h)-1}\,b_1(s(t))\,s'(t). \]
	Since $y_1\in \Sol$ is not constant, $b_1(s(t))$ is not zero.
	Therefore,
	\begin{equation}\label{eq:diffS}
	n\,s(t)^{n(1-h)-1}\,s'(t)=m\,t^{m(1-h)-1}.
	\end{equation}
	Finally, comparing orders, since $h \neq 1$ by assumption, we get that $n=m$.
	
	Assume now that $\Delta(y_1)=\Delta(y_2)$.
	Then, $\delta(y_1)=\delta(y_2)$ and hence, $y_1(t^n)=y_2(t^n)$.
	Thus, $y_1(x)=y_2(x)$.
\end{proof}

\begin{definition}
	The \textit{ramification order} of a solution parametrization $A(t)$ is defined as the ramification order of $\Delta^{-1}(A(t))$.
\end{definition}

In the following we analyze the number of solution parametrizations in a solution place.
We start with a technical lemma.

\begin{lemma}\label{lemma:laurent}
	Let $a(t)\in \C((t))$ be non-constant and let $\alpha_1,\alpha_2\in \C$ be two different $k$-th roots of unity.
	If $a(\alpha_1 t)=a(\alpha_2 t)$, then there exists a minimal $m\in \mathbb{Z}_{>1}$, with $m \leq k$, such that $a(t)$ can be written as $a(t)=\sum_{j\geq j_0/m} a_{jm}\,t^{jm}$.
\end{lemma}
\begin{proof}
	Let $a(t)=\sum_{j \geq j_0} a_j\,t^j$.
	Since $a(\alpha_1 t)=a(\alpha_2 t)$, then $a_j \alpha_{1}^{j}=a_j \alpha_{2}^{j}$.
	So, if $a_j\neq 0$ then $(\alpha_1/\alpha_2)^{j}=1$.
	Let $m \in \mathbb{Z}_{>1}$ be minimal such that $\alpha_1/\alpha_2$ is an $m$-th primitive root of unity.
	Then $(\alpha_1/\alpha_2)^{j}=1$ if and only if $j$ is a multiple of $m$ and this implies that $a(t)=\sum_{j\geq j_0/m} a_{jm}\,t^{jm}$.
\end{proof}

\begin{lemma}\label{lemma:number of solution parametrizations}
	Let $[A]$ be solution place of ramification order $n$. 
	It holds that 
	\begin{enumerate}
		\item If $h \leq 0$, then there are exactly $n(1-h)$ solution parametrizations in $[A]$, $A(t)$ is a solution parametrization, and all solution parametrizations in the place are of the form $A(\alpha\,t)$ where $\alpha^{n(1-h)}=1$.
		\item If $h \geq 2$, then there are infinitely many solution parametrizations in $[A]$.
	\end{enumerate}
\end{lemma}
\begin{proof}
	Let, for $i=1,2$, $(a_i,b_i)\in [A]$ be two different solution pa\-ra\-me\-tri\-za\-tions.
	As a consequence of equation \eqref{eq:diffS}, we get that the order one formal power series $s(t)$ relating $(a_1,b_1)$ and $(a_2,b_2)$ satisfies
	\begin{equation} \label{eq-help}
	s(t)^{n(1-h)-1}\,s'(t)=t^{n(1-h)-1},
	\end{equation} where $n$ is the ramification order of the place.
	Conversely, let $s(t)$ be a solution of \eqref{eq-help} with $\ord_t(s(t))=1$ and $(a_3(t),b_3(t))=(a_1(s(t)),b_1(s(t))$.
	Then
	\[ a_3'(t)=(a_1(s(t)))'=a_1'(s(t))\,s'(t)=n\,s(t)^{n(1-h)-1}\,b_1(s(t))\,s'(t), \]
	and by using equation \eqref{eq-help},
	\[a_3'(t)=n\,t^{n(1-h)-1}\,b_3(t). \]
	Then, by Proposition \ref{Theorem:Characterization Sol Param}, $(a_3(t),b_3(t))=(a_1(s(t)),b_1(s(t))$ is a solution paramerization.
	
	Let us compute the solutions of \eqref{eq-help} by separation of variables.
	If $h \leq 0$, then $s(t)=\alpha\,t$, where $\alpha^{n(1-h)}=1$.
	Therefore, the set of all solution parametrizations in $[A]$ is $$\mathcal{A}:=\{(a_1(\alpha\,t),b_1(\alpha\,t))\, |\,\alpha^{n(1-h)}=1\}.$$
	Let us verify that $\#(\mathcal{A})=n(1-h)$.
	If $n(1-h)=1$, the result is trivial.
	Let $n(1-h)>1$, and let us assume that $\#(\mathcal{A})<n(1-h)$.
	Then, there exist two different $n(1-h)$-th roots of unity, $\alpha_1,\alpha_2$, such that $(a_1(\alpha_1 t),b_1(\alpha_1 t))=(a_1(\alpha_2 t),b_1(\alpha_2 t))$.
	By Lemma \ref{lemma:laurent} there exists $m\in \mathbb{Z}_{>1}$ with $m\leq n(1-h)$, such that $a_1(t),b_1(t)$ can be written as $a_1(t)=\sum_{j\geq j_0/m} c_{jm}\,t^{jm}$ and $b_1(t)=\sum_{j\geq k_0/m} d_{jm}\,t^{jm}$.
	This implies that $(a_1,b_1)$ is reducible, which is a contradiction.
	
	If $h \geq 2$, the solutions of \eqref{eq-help} are of the form $$s(t)=\frac{\alpha\,t}{\sqrt[n(h-1)]{1+t^{n(h-1)}\,c}},$$ where $c$ is an arbitrary constant and $\alpha^{n(h-1)}=1$.
	Note that $s(t)$ can indeed be written as a formal power series of first order and for every choice $c \in \C$ the solution parametrization is distinct.
\end{proof}

For a given parametrization $(a(t),b(t)) \in \Param$ satisfying \eqref{eq:nec_eq_order_ramification}, our strategy for finding the solutions will be to determine  reparametrizations fulfilling \eqref{eq:necessary_DE}, i.e. we are looking for $s(t) \in \C[[t]]$ with $\ord_t(s(t))=1$ such that $(a(s(t)),b(s(t))$ satisfies
\begin{equation*}
\frac{d\,(a(s(t)))}{dt}=n\,t^{n(1-h)-1}\,b(s(t)).
\end{equation*}
By applying the chain rule, this is equivalent to the \textit{associated differential equation}
\begin{equation}\label{eq:DE_2}
a'(s(t))\cdot s'(t)=n\,t^{n(1-h)-1}\,b(s(t)).
\end{equation}

In the following lemmas we analyze the solvability and other properties of solutions of the associated differential equation.

\begin{lemma}[Briot-Bouquet]\label{le:Briot-Bouquet}
	Let $g(t,z),f(t,z)\in \mathbb{L}[[t,z]]$, where $\mathbb{L}$ is a subfield of the complex numbers $\mathbb{C}$.
	Let us consider the differential equation
	\begin{equation} \label{eq:Briot-Bouquet}
	g\big(t,z(t)\big)\,\,t\,z'(t) = f(t,z(t))
	\end{equation}
	with $g(0,0)\neq 0$ and $f(0,0)=0$.
	Let us denote $$\lambda=\frac{1}{g(0,0)}\frac{\partial f}{\partial z}(0,0),$$ and let $\Sigma$ be the set of formal power series solutions of equation \eqref{eq:Briot-Bouquet} in $\mathbb{C}[[t]]$ of order greater or equal to one. 
	Then it holds that
	\begin{enumerate}
		\item If $\lambda$ is not a positive integer, then $\Sigma$ has exactly one element $z(t)=\sum_{i=1}^{\infty}\zeta_i\,t^{i}$ and all the coefficients $\zeta_i\in \mathbb{L}$.
		\item If $\lambda$ is a positive integer, then $\Sigma$ is either the empty set or a one-parameter family of the form $z(t)=\sum_{i=1}^{\infty}\zeta_i\,t^{i}$,
		where $\zeta_1,\ldots,\zeta_{\lambda-1}$ are uniquely determined elements in $\mathbb{L}$,
		$\zeta_\lambda\in \mathbb{C}$ is a free parameter and for $i>\lambda$ the coefficients $\zeta_i\in \mathbb{L}(\zeta_\lambda)$ are determined by $\zeta_\lambda$.
		Moreover, it can algorithmically be decided whether $\Sigma$ is empty or not and two elements in $\Sigma$ are equal if they coincide up to order $\lambda$.
	\end{enumerate}
	Furthermore, the following statements hold in both cases:
	\begin{enumerate}
		\item[(a)] If $g$ and $f$ are convergent power series, then any element of $\Sigma$ is convergent.
		\item[(b)] Let us write $f=\sum_{i+j \geq 1} f_{i,j}\,t^{i}\,z^{j}$, $g=\sum_{i+j \geq 0} g_{i,j}\,t^{i}\,z^{j}$, $f_{i,j},g_{i,j}\in \mathbb{L}$ and let	$z(t)=\sum_{i=1}^{\infty}\zeta_i\,t^{i}$ be an element of $\Sigma$. 
		For any $m\in \mathbb{Z}_{>0}$, the coefficient $\zeta_m$ is completely determined by the coefficients of $f_{i,j}$ and $g_{i,j}$ such that $i+j\leq m$ and, in case $\lambda \in \mathbb{Z}_{>0}$ and $m>\lambda$, $\zeta_m$ also depends on the coefficient $\zeta_\lambda$.
	\end{enumerate}
\end{lemma}
\begin{proof}
	These results are partly consequences of the Theorem XXVIII in Section 80 and Section 86 in \cite{BriotBouquet-Reserches}. 
	Nevertheless we will prove existence and uniqueness of the solutions independently in order to show precisely the dependency of the coefficients of the solutions with respect to the coefficients of $f(t,z)$ and $g(t,z)$.
	
	We may assume that $g(0,0)=1$ and $\frac{\partial\,f}{\partial z}(0,0)=\lambda$ by multiplying both sides of equation~\eqref{eq:Briot-Bouquet} with $1/g(0,0)$. 
	Let us write $g=1-\tilde{g}$ and $f=\lambda\,z+f_{1,0}\,t+\tilde{f}$, where $\tilde{g}=\sum_{i+j\geq 1}g_{i,j}\,t^{i}\,z^{j}$ and $\tilde{f}=\sum_{i+j\geq2}f_{i,j}\,t^{i}\,z^{j}$. 
	Then, the differential equation~\eqref{eq:Briot-Bouquet} is equivalent to the following one
	\begin{equation} \label{eq:Briot-Bouquet-simplify}
	t\, z'(t)-\lambda \,z=f_{1,0}\,t+\tilde{f}(t,z)+ t\,z'(t) \,\tilde{g}(t,z). 
	\end{equation}
	
	Let us substitute $z(t)$ for an arbitrary power series $\sum_{i=1}^{\infty}\zeta_i\,t^{i}$ into the equation~\eqref{eq:Briot-Bouquet-simplify}. 
	For $m \in \mathbb{Z}_{>0}$ let us compute the coefficient of $t^{m}$ in both sides of the resulting expression. 
	On the left hand side we obtain $(m-\lambda )\zeta_m$. 
	By expanding the right hand side of~\eqref{eq:Briot-Bouquet-simplify}, we obtain a polynomial expression in $t,\zeta_i, f_{i,j}$ and $g_{i,j}$ with non-negative integer coefficients. 
	More precisely, since $\ord_t \,z(t) \geq 1$, the coefficient of $t^{m}$ in $t\,z'(t)\,\tilde{g}(t,z(t))$ is a polynomial expression in $\zeta_{i}$, $1 \leq i \leq m-1$, and $g_{i,j}$ with $1 \leq i+j \leq m-1$.  
	The coefficient of $t^m$ in $\tilde{f}(t,z(t))$ is a polynomial expression in $\zeta_i$, $1 \leq i \leq m-1$ and $f_{i,j}$ for $2 \leq i+j \leq m-1$. 
	Let us define $P_m$ as the coefficient of $t^m$ of the right hand side of~\eqref{eq:Briot-Bouquet-simplify}. 
	Then $z(t)$ is a solution of equation~\eqref{eq:Briot-Bouquet-simplify} if and only if for every $m \in \mathbb{Z}_{>0}$ the following relations hold
	\begin{equation} \label{eq:Brior-Bouquer_recursive_equation}
	(m-\lambda)\,\zeta_m=P_m({\zeta}_1,\ldots,{\zeta}_{m-1},{f}_{1,0},g_{1,0},g_{0,1},{f}_{i,j},{g}_{i,j};~2 \leq i+j\leq m-1).
	\end{equation}
	If $\lambda\not\in\mathbb{Z}_{>0}$, equation~\eqref{eq:Brior-Bouquer_recursive_equation} can be solved uniquely for every $\zeta_m$, $m\geq 1$. 
	If $\lambda\in \mathbb{Z}_{>0}$, then $\zeta_1,\ldots,\zeta_{\lambda-1}$ satisfying equations~\eqref{eq:Brior-Bouquer_recursive_equation} are uniquely determined as well. 
	If equation~\eqref{eq:Briot-Bouquet-simplify} has at least one formal power series solution of order greater than or equal to one, 
	then necessarily $P_{\lambda}(\zeta_1,\ldots,\zeta_{\lambda-1},f_{i,j},g_{i,j})=0$ and arbitrary $\zeta_{\lambda}$ satisfies	equation~\eqref{eq:Brior-Bouquer_recursive_equation}. 
	Once that $\zeta_\lambda$ has been chosen, there exists a unique sequence of $\zeta_{i}$, $i>\lambda$, solving the equation~\eqref{eq:Brior-Bouquer_recursive_equation} for $m>\lambda$. 
	Since equation~\eqref{eq:Brior-Bouquer_recursive_equation} is linear in $\zeta_m$ and $P_m$ is a polynomial expression, no field extensions are necessary. 
	
	In order to show item (a), let $f$ and $g$ be convergent power series in a neighborhood of the origin. 
	In the case $\lambda\not\in \mathbb{Z}_{>0}$ the convergence of the solution follows by the majorant series method using the fact that the coefficients of $P_m$ are non negative (see for instance Section 12.6 in \cite{Ince1926}). 
	In the case $\lambda\in \mathbb{Z}_{>0}$ we perform for a solution $z(t)=\sum \zeta_i\,t^{i}$ the change of variables $z(t)=\zeta_1t+\cdots+\zeta_{\lambda}t^{\lambda}+t^{\lambda}\,w(t)$ and reduce it to the previous case, see for instance Section 86 of \cite{BriotBouquet-Reserches}.
\end{proof}

\begin{lemma}\label{le:existence_and_convegence_of_reparametriztion}
	Let $\qq=(y_0,p_0) \in \C \times \C_{\infty}$ and $\mathbb{L}$ be a subfield of $\C$.
	Let $\cP=[(a(t),b(t))] \in \Places$ with $a(t),b(t) \in \mathbb{L}((t))$ be such that equation~\eqref{eq:nec_eq_order_ramification} holds for an $n \in \mathbb{Z}_{>0}$, i.e. $n=\frac{k-r}{1-h}\geq 1$ with $k=\ord_t(a(t)-y_0),~r=\ord_t(b(t))$.
	Let $\Sigma$ be the set of formal power series solutions of the associated differential equation~\eqref{eq:DE_2} in $\mathbb{C}[[t]]$ of order one. 
	Then, it holds that
	\begin{enumerate}
		\item If $h\leq 0$, then $\Sigma$ consists of exactly $n\,(1-h)$ elements of the form $s(t)=\sum_{i=1}^{\infty} \sigma_i\,t^i$, where $\sigma_1^{n(1-h)} \in \mathbb{L}$ and all the other coefficients $\sigma_i \in \mathbb{L}(\sigma_1)$.
		\item If $h\geq 2$, then $\Sigma$ is either the empty set or consists of up to $n\,(h-1)$ one-parameter families of the form $s(t)=\sum_{i=1}^{\infty} \sigma_i\,t^i$, where $\sigma_1^{n(h-1)} \in \mathbb{L}$, $\sigma_2,\ldots,\sigma_{r-k-1}$ are uniquely determined elements in $\mathbb{L}(\sigma_1)$, $\sigma_{r-k} \in \C$ is a free parameter and for $i>r-k$ the coefficients $\sigma_{i} \in \mathbb{L}(\sigma_1,\sigma_{r-k})$ are determined by $\sigma_1$ and $\sigma_{r-k}$.
	\end{enumerate}
	Moreover, the following statements hold in both cases
	\begin{enumerate}
		\item[(a)] If $a(t)$ and $b(t)$ are convergent as Puiseux series, then any element in $\Sigma$ is convergent.
		\item[(b)] Let us write $a(t)=y_0+\sum_{i \geq 0} a_i\,t^{k+i},~b(t)=\sum_{i \geq 0}b_i\,t^{r+i}$ with $a_i,b_i \in \mathbb{L}$.
		Then for any $m \in \mathbb{Z}_{>1}$, the coefficient $\sigma_m$ is completely determined by the coefficients $\sigma_1,a_i$ and $b_i$ such that $0 \leq i \leq m-1$ and, in case $r-k \in \mathbb{Z}_{>0}$ and $m>r-k$, $\sigma_m$ also depends on the coefficient $\sigma_{r-k}$.
	\end{enumerate}
\end{lemma}
\begin{proof}
	Let us define $\nu=|n\,(1-h)|=|k-r| \geq 1$.
	
	First we prove the case $h \leq 0$.
	Multiplying both sides of \eqref{eq:DE_2} by $s(t)^{-r}\,t^{-\nu+1}$, we obtain the equivalent differential equation
	\begin{equation}\label{eq:DE_3}
	t^{-\nu+1}\,\tilde{a}(s(t))\cdot s'(t)=n\,\tilde{b}(s(t)),
	\end{equation}
	where $\tilde{a}(s)=s^{-r}\,a'(s)=\sum_{i \geq 0}\tilde{a}_{i}\,s^{i+\nu-1}$ with $\tilde{a}_i=(k+i)\,a_{i},~\tilde{a}_{0}=k\,a_k \neq 0$ and $\tilde{b}(s)=s^{-r}\,b(s)=\sum_{i \geq 0}b_i\,s^{i}$. 
	Let us fix a non-zero $\sigma\in \mathbb{C}$ and perform the change of variables $s(t)=t\,(\sigma+z(t))$ in the differential equation~\eqref{eq:DE_3} to obtain
	\begin{equation*}
	\sum_{i \geq 0}\tilde{a}_i\,t^{i}\,(\sigma+z(t))^{i+\nu-1}\,(\sigma+z(t)+t\,z'(t)) =n\,\sum_{i \geq 0}b_i\,t^{i}\,(\sigma+z(t))^{i}.
	\end{equation*}
	Let us move the terms of the left hand side not involving $z'$ to the right hand side to obtain
	\begin{equation} \label{eq:DE_BriotBouquet}
	\left( \sum_{i \geq 0}\tilde{a}_{i}\,t^{i}\,(\sigma+z(t))^{i+\nu-1}\right)\,t\,z'(t)= \sum_{i \geq 0}t^{i}\,\Big( n\,b_i\,(\sigma+z(t))^{i}-\tilde{a}_{i}\,(\sigma+z(t))^{i+\nu}\Big).
	\end{equation}
	Hence, a formal (respectively convergent) power series $s(t)=\sum_{i=1}^{\infty} \sigma_i\,t^{i}$ is a solution of~\eqref{eq:DE_2} if and only if $z(t)=\sum^{\infty}_{i=2}\sigma_{i}t^{i-1}$ is a formal (respectively convergent) power series solution of~\eqref{eq:DE_BriotBouquet} for $\sigma=\sigma_1$.
	
	By considering in equation~\eqref{eq:DE_BriotBouquet} the terms independent of $t$, we obtain $0=n\,b_0-\tilde{a}_{0}\,\sigma^{\nu}$. 
	As a consequence, if $s(t)=\sum_{i=1}^{\infty}\sigma_i\,t^{i}$ is a formal power solution of equation~\eqref{eq:DE_2}, then $\sigma_1^{\nu}=\frac{n\,b_0}{\tilde{a}_{0}}$. 
	Since $\tilde{a}_{0},~b_0 \neq 0$, there are exactly $\nu$ possibilities for $\sigma_1$. 
	It remains to prove that for any such $\sigma_1=\sigma$, there exists a unique formal power series solution $z(t)$ of~\eqref{eq:DE_BriotBouquet} with $\ord_t z(t)\geq 1$ satisfying the properties specified in the statement of the lemma. 
	As described below, this is a direct consequence of Lemma~\ref{le:Briot-Bouquet} applied to the differential equation~\eqref{eq:DE_BriotBouquet}.
	
	First, let us show that~\eqref{eq:DE_BriotBouquet} satisfies the hypothesis of the first case of
	Lemma~\ref{le:Briot-Bouquet}. 
	Let us denote by $f(t,z(t))$ the right hand side of equation~\eqref{eq:DE_BriotBouquet} and by $g(t,z(t))\,t\,z'(t)$ the left hand side. 
	We have that $g(0,0)=\tilde{a}_{0}\,\sigma_1^{\nu-1} \neq 0,~f(0,0)=n\,b_0-\tilde{a}_{0}\sigma_1^{\nu}=0$ and $\frac{\partial\,f}{\partial z}(0,0)=-\nu\,\tilde{a}_{0}\sigma_1^{\nu-1}$. 
	Hence $\frac{1}{g(0,0)}\frac{\partial\,f}{\partial z}(0,0)=-\nu=-n\,(1-h)<0$. 
	Then, by Lemma~\ref{le:Briot-Bouquet}, the following statements hold. 
	There exists a unique formal power series solution $z(t)=\sum_{i=2}\sigma_it^{i-1}$ of~\eqref{eq:DE_BriotBouquet} and consequently, $s(t)=\sum_{i=1}\sigma_i\,t^{i}$ is a solution of~\eqref{eq:DE_2}.
	If $a(s)$ and $b(s)$ are convergent, the series $f(t,z)$ and $g(t,z)$ are convergent and then
	$z(t)$ and $s(t)$ are convergent. 
	Moreover, the coefficients of $f(t,z),g(t,z)$ and therefore of $z(t)$ and $s(t)$ belong to the field $\mathbb{L}(\sigma_1)$.
	
	It remains to prove item (b) for the case of $h \leq 0$. 
	Since $\nu-1 \geq 0$, the coefficient $g_{i,j}$ in $g(t,z)=\sum_{i+j \geq 0} g_{i,j}\,t^{i}\,z^{j}$ depends only on $\sigma_1$ and $\tilde{a}_{i}$. 
	Similarly, the coefficient $f_{i,j}$ in $f(t,z)=\sum_{i+j \geq 1} f_{i,j}\,t^{i}\,z^{j}$ depends only on $\sigma_1,~\tilde{a}_{i}$ and $b_i$. 
	For $m \in \mathbb{Z}_{>0}$, by Lemma~\ref{le:Briot-Bouquet}, $\sigma_m$ depends only on $f_{i,j}$ and $g_{i,j}$ with $i+j\leq m-1$, which in their turn depend on $\sigma_1$ and $\tilde{a}_{i}$ and $b_i$ for $0\leq i \leq m-1$.
	Since $\tilde{a}_{i}=(k+i)\,a_i$, item (b) is proven.

	Let us now consider the case $h \geq 2$. 
	Multiplying both sides of~\eqref{eq:DE_2} by $s(t)^{-r}$, we obtain the equivalent differential equation
	\begin{equation} \label{eq:Lemma8_hpositive}
	\tilde{a}(s(t))\,s'(t)=n\,t^{\nu+1}\,\tilde{b}(s(t)),
	\end{equation}
	where $\tilde{a}(s)=s^{-k+1}\,a'(s)=\sum_{i \geq 0}\tilde{a}_i\,s^{i}$ with $\tilde{a}_i=(k+i)\,a_{i},~\tilde{a}_0=k\,a_0 \neq 0$ and $\tilde{b}(s)=s^{-k+1}\,b(s)=\sum_{i \geq 0}b_i\,s^{i+\nu+1}$. 
	After performing for $\sigma \in \C$ the change of variable $s(t)=t\,(\sigma+z(t))$ in~\eqref{eq:Lemma8_hpositive}, we obtain the equivalent differential equation
	\begin{equation} \label{eq:DE_BriotBouquet-h-positive}
	\left(\sum_{i \geq 0}\tilde{a}_i\,t^{i}\,(\sigma+z(t))^{i}\right)\,t\,z'(t)=\sum_{i \geq 0}t^{i}\,\Big(n\,b_{i}\,(\sigma+z(t))^{i+\nu+1}-\tilde{a}_i\,(\sigma+z(t))^{i+1}\Big).
	\end{equation}
	By considering in equation \eqref{eq:DE_BriotBouquet-h-positive} the terms independent of $t$, we again obtain $\sigma_1^{\nu}=\frac{\tilde{a}_0}{n\,b_0}$ as necessary condition for a solution $s(t)=\sum_{i=1}^{\infty} \sigma_i\,t^i$ of~\eqref{eq:DE_2}. 
	For $\sigma=\sigma_1$, by setting $f(t,z)$ equal to the right hand side and $g(t,z)\,t\,z'$ equal to the left hand side in equation~\eqref{eq:DE_BriotBouquet-h-positive}, the hypothesis of the second case of Lemma~\ref{le:Briot-Bouquet} are fulfilled: $g(0,0)=\tilde{a}_0 \neq 0,~f(0,0)=0,~\frac{\partial\,f}{\partial z}(0,0)=n\,(\nu+1)\,b_{0}\,\sigma_1^{\nu}-\tilde{a}_{0}$ and $\lambda=\frac{1}{g(0,0)}\frac{\partial\,f}{\partial z}(0,0)=\nu\in \mathbb{Z}_{>0}$. 
	As a consequence of Lemma~\ref{le:Briot-Bouquet} and the equivalence between equations~\eqref{eq:DE_2} and~\eqref{eq:DE_BriotBouquet-h-positive}, we obtain that for every $\sigma_1$ with $\sigma_1^{\nu}=\frac{\tilde{a}_0}{n\,b_0}$ equation~\eqref{eq:DE_2} either has no formal power series solutions of order one or it has a one-parametric family. 
	In the affirmative case, two of the solutions are equal if they coincide up to order $\nu+1$. 
	The others properties of these solutions are proven analogously as in the preceding case.
\end{proof}

Now, in the case of non-positive $h$, we are in the position to decide whether a given place $\cP\in \Places$ is a solution place by a simple order comparison.

\begin{theorem}\label{theorem:characterizationSolPlace}
	Let $\cP=[(a(t),b(t))]\in \Places$ and $h \leq 0$.
	Then $\cP$ is a solution place if and only if equation \eqref{eq:nec_eq_order_ramification} holds for an $n \in \mathbb{Z}_{>0}$.
	In the affirmative case the ramification order of $\cP$ is equal to $n$.
\end{theorem}
\begin{proof}
	The first direction directly follows from Lemma \ref{lemma:nec_cond_DE}.
	For the other direction, let $(a(t),b(t))$ and $n \in \mathbb{Z}_{>0}$ be such that equation \eqref{eq:nec_eq_order_ramification} holds.
	Moreover, let $s(t) \in \C[[t]]$ with $\ord_t(s(t))=1$ be such that the associated differential equation \eqref{eq:DE_2} is fulfilled. 
	By Lemma \ref{le:existence_and_convegence_of_reparametriztion}, such a solution exists. 
	Then $(\bar{a}(t),\bar{b}(t))=(a(s(t)),b(s(t))$ fulfills equation \eqref{eq:necessary_DE} and by Proposition \ref{Theorem:Characterization Sol Param}, $(\bar{a}(t),\bar{b}(t))$ is a solution parametrization with ramification order equal to $n$.
\end{proof}

\begin{theorem}
	Any formal Puiseux series solution of \eqref{eq-main}, expanded around a finite point or at infinity, is convergent.
\end{theorem}
\begin{proof}
	In order to prove the statement we show that every formal Puiseux series solution of equation \eqref{eq-infinity}, in particular for $h \in \{0,2\}$, expanded around zero is convergent.
	Let $y(x) \in \Sol$.
	Performing the change of variable $\tilde{y}(x)=1/y(x)$ if necessary, we can assume that $y_0 \in \C$.
	Let $n \in \mathbb{Z}_{>0}$ be the ramification order of $y(x)$ and $\Delta(y(x))=(a(t),b(t))$.
	By Lemma \ref{lemma:nec_cond_DE}, equations \eqref{eq:necessary_DE} and \eqref{eq:nec_eq_order_ramification} hold.
	Let $k=\ord_t(a(t)-y_0)\geq 1$.
	By Section 2 of Chapter IV in \cite{walker1950algebraic}, there exists a formal power series $s(t) \in \C[[t]]$, with $\ord_t(s(t))=1$, such that $$a(s(t))-y_0=t^k.$$
	Let $\bar{a}(t)=a(s(t))$ and $\bar{b}(t)=b(s(t))$.
	Then $(\bar{a}(t)-y_0,\bar{b}({t}))=(t^k,\bar{b}(t))$ is a local parametrization of the non-trivial algebraic curve defined by $F(y-y_0,p)$.
	Hence, by Puiseux's theorem, $\bar{b}(t)$ is convergent.
	
	Let $r(t)$ be the compositional inverse of $s(t)$, i.e. $r(s(t))=t=s(r(t)).$
	Then $r(t)$ is a formal power series of order one and $a(t)=\bar{a}(r(t)), b(t)=\bar{b}(r(t))$.
	Since equation \eqref{eq:DE_2} holds for $(\bar{a}(t),\bar{b}(t))$ and $r(t)$, by Lemma \ref{le:existence_and_convegence_of_reparametriztion}, $r(t)$ is convergent.
	This implies that $a(t)$ is convergent and therefore, $y(x)=a(x^{1/n})$ is convergent as a Puiseux series.
\end{proof}

\begin{theorem} \label{tm:existence}
	Let $F(y,p)$ be a non-constant polynomial with no factor in $\C[y]$ or $\C[p]$.
	For any point in the plane $(x_0,y_0)\in \C^2$, there exists an analytic solution $y(x)$ of $F(y,y')=0$ such that $y(x_0)=y_0$.
\end{theorem}
\begin{proof}
	It is sufficient to prove the existence of a convergent formal Puiseux series solution $y(x)=y_0+\sum_{i \geq 1} c_i\,(x-x_0)^{i/n}$.
	Performing the change of variable $\bar{x}=x-x_0$ and $\bar{y}=y-y_0$, we may assume that $x_0=0$ and $y_0=0$.
	
	Let us write $F(y,p)=\sum F_{i,j}\,y^{i}\,p^{j}$.
	If $F(0,0)=F_{0,0}=0$, then we have that $y(x)=0$ is a solution of $F(y,y')=0$ and $\Delta(y)$ passes through $(0,0)$.
	We may assume that $F_{0,0}\neq 0$. Consider $\mathcal{N}(F)$ the Newton polygon of the algebraic curve $F(y,p)=0$ in the variables $y$ and $p$.
	The point $(0,0)$ is a vertex of $\mathcal{N}(F)$, because $F_{0,0}\neq 0$.
	This implies that all the sides of $\mathcal{N}(F)$ have slope greater or equal to zero (see Figure~\ref{fig:2}).
	
	\begin{figure}[H]
		\centering
		\def\varRightEnd{5}
\def\varLeftEnd{-1}
\def\varTopEnd{5}
\def\varBottomEnd{0}
\def\varMargen{0.4}
\def\varRadius{0.1}

\begin{tikzpicture}[scale=0.7]
  \path[fill=yellow] ({\varRightEnd+\varMargen},0) -- (0,0) -- (0,2) -- (1,4) --
  (3,5) -- ({\varRightEnd+\varMargen},5)-- ({\varRightEnd+\varMargen},0);
  \draw [thick]  (0,0) -- (0,2) -- (1,4) -- (3,5)--({5+\varMargen},5);
  \draw[help lines, dotted] [ultra thin]
  ({\varLeftEnd-\varMargen},{\varBottomEnd-\varMargen}) grid  
  ({\varRightEnd+\varMargen},{\varTopEnd+\varMargen});
  \draw[->] [thick] ({\varLeftEnd-\varMargen},0) --
  ({\varRightEnd+\varMargen},0);
  \draw[->] [thick] (0,{\varBottomEnd-\varMargen}) -- (0,{\varTopEnd+\varMargen});
  \draw[thick, fill=lightgray] (0,0) circle [radius={\varRadius}];
  \draw[thick, fill=lightgray] (0,2) circle [radius={\varRadius}];
  \draw[thick, fill=lightgray] (1,4) circle [radius={\varRadius}];
   \draw[thick, fill=lightgray] (3,5) circle [radius={\varRadius}];
  \node[below] at ({5+\varMargen},-0.2) {$y$}; 
  \node[left] at (-0.2,5) {$p$};
\end{tikzpicture}

		\caption{The Newton polygon of the algebraic curve $F(y,p)=0$. All its sides have non-negative slope, because the point $(0,0)\in \mathcal{N}(F)$.}
		\label{fig:2}
	\end{figure}
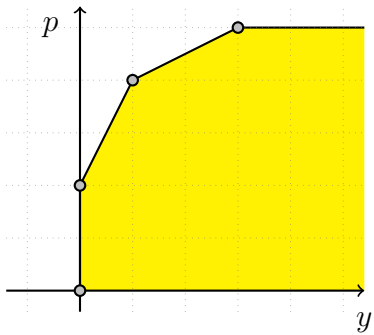
	
	Since the degree of $F(y,p)$ with respect to $p$ is positive, $\mathcal{N}(F)$ has at least one side.
	Therefore, by Puiseux's Theorem, there exists a convergent Puiseux series solution $p(y)$ of the algebraic equation $F(y,p(y))=0$ of the form $p(y)=\sum_{i \geq k} c_i\,y^{i/n}$, where $c_k\neq 0$ and $k\leq 0$.
	Let us define $a(t)=t^{n}$ and $b(t)=\sum_{i=k}^{\infty} c_i\,t^{i}$.
	Then $(a(t),b(t))$ is a convergent parametrization of $\cc$ satisfying
	\begin{displaymath}
	m=\ord_t(a(t)-a(0))-\ord_t(b(t))=n-k\geq n\geq 1.
	\end{displaymath}
	By Theorem \ref{theorem:characterizationSolPlace}, there exists a formal Puiseux series solution $y(x)$ of the differential equation $F(y,y')=0$ and $\ord_x(y(x))>0$ which proves the theorem.
\end{proof}

Notice that in Theorem \ref{tm:existence} we can give a lower and an upper bound for the number of solution parametrizations passing through a given point $(x_0,y_0) \in \C^2$.
First, every side with slope greater or equal to zero defines a different solution parametrization.
Thus, a lower bound can easily be derived after computing the Newton polygon $\mathcal{N}(F)$.

Second, let $\Sigma_{(x_0,y_0)}$ denote the set of solution parametrizations passing through $(x_0,y_0)$.
The set of corresponding solution places is denoted by $\cP(y_0)=\{[(a(t),b(t))]~|~ (a(t),b(t)) \in \Sigma_{(x_0,y_0)} \}$.
Since every solution parametrization passing through $(x_0,y_0)$ is a solution parametrization centered at $(y_0,p_0)$ for some $p_0 \in \C_{\infty}$, by Lemma \ref{lemma:number of solution parametrizations}, $$\#\Sigma_{(x_0,y_0)}=\sum_{P \in \cP(y_0)} \text{ramification order of } P \leq \deg_p(F).$$
The last inequality is a well known result for algebraic curves and can be found for example in \cite{Duval1989}[Theorem 1].

As a consequence for example the family of functions $$y(x)=x+c\,x^2,$$ where $c$ is an arbitrary constant, cannot be a solution of any first order autonomous ordinary differential equation.
Otherwise, there are infinitely many distinct formal parametrizations $(y(x),y'(x))$ with $y_0=0$ as initial value and the sum of the ramification orders of $P \in \cP(y_0)$ is infinite in contradiction to the bound above.

We note that there might be families of formal Puiseux series solutions at infinity for an autonomous first order ordinary differential equation as we will see in Example \ref{ex-infinity}.

\section{Algorithms and Examples}\label{sec-algorithms}

In this section we outline an algorithm that is derived from the results in Section \ref{sec-PSP}, in particular, for $h \in \{0,2\}$.
We can describe algorithmically all formal Puiseux series solutions of the differential equation \eqref{eq-main}.
For each formal Puiseux series solution we will provide what we call a \textit{determined solution truncation}.
A determined solution truncation is an element of $\C[x^{1/n}][x^{-1}]$, for some $n \in \mathbb{Z}_{>0}$, that can be extended uniquely to a formal Puiseux series solution.

\para

If $F$ is reducible, one could factor it and consider its irreducible components and the solutions of the corresponding differential equations.
However, from a computational point of view, this is not optimal, and we compute the square-free part of $F$ instead.
So let us assume $F \in \C[y,p]$ to be square-free and have no factor in $\C[y]$ or $\C[p]$ in the remaining of the paper.
Since each formal Puiseux series solution $y(x)$ gives rise to an initial tuple $\qq=(y(0),(x^hy'(x))(0))$ in $\cc$, we will describe for each point $\qq \in \cc$ the set $\Sol$.
We note that if $\ord_x(y) \geq 0$ and $h \geq 2$, then $\qq$ will necessarily be of the type $(y_0,0)$ for some $y_0 \in \C$.

\subsection{Solutions expanded around zero} \label{subsec-zero}

In this subsection we consider formal Puiseux series solutions of \eqref{eq-main}, or equivalently, solutions of \eqref{eq-infinity} with $h=0$ expanded around zero.
A point $\qq=(y_0,p_0) \in \cc$ is called a \textit{critical curve point} if either $p_0 \in \{0,\infty\}$ or $\frac{\partial F}{\partial p}(\qq)=0$ (compare \cite{FalkensteinerSendra_2018}).
Under our assumptions, the set of critical curve points, denoted by $\mathcal{B}(F)$, is finite. 

If $\qq=(y_0,p_0) \in \cc \setminus \mathcal{B}(F)$, we can apply the method of limits (see Chapter XII in \cite{Ince1926}).
The only formal Puiseux series solution with $\qq$ as initial tuple is a formal power series and its determined solution truncation is given by $y_0+p_0x$.

The points $\qq=(y_0,\infty) \in \cc$ with $y_0 \in \C$, can be computed by considering $F \in \C[y][p]$ and determining the zeros of the leading coefficient in $y$.
As already remarked in Section \ref{sec:preliminaries}, the possible curve point $(\infty,\infty)$ can be handled by a suitable change of variables.
Note that there cannot be a solution with an initial tuple of the form $(\infty, p_0)$ with $p_0 \in \C$, because if $\ord_x(y(x))<0$ then $\ord_x(y'(x))<0$ as well.

\para

Assume that $\qq \in \mathcal{B}(F)$ is a critical curve point.
Let $\RT \subseteq \C[t][t^{-1}]$ denote the set of truncations of non-equivalent classical Puiseux para\-metrizations $(y_0+t^k,b(t)) \in \Param$, where the first $N$ terms of $b(t)$ are computed.
In \cite{Duval1989} is presented an algorithm to compute $\RT$ and with $N$ equal to $2(\deg_p(F)-1) \deg_y(F)+1$ or the Milnor number (see \cite{RISC4119}) is given a bound for the truncation such that $\RT$ is in one-to-one correspondence to $\Places$.
Moreover, then the ramification orders of the approximated places are determined such that we can check whether equation \eqref{eq:nec_eq_order_ramification} holds. 
Following the proof of Lemma \ref{le:existence_and_convegence_of_reparametriztion}, we can describe all formal Puiseux series solutions with $\qq$ as initial tuple as Algorithm PuiseuxSolve shows.
By choosing the bound $N=2(\deg_p(F)-1) \deg_y(F)+1$ no further extensions of the ground field for computing the coefficients are required.

\begin{algorithm}[H]
	\caption{PuiseuxSolve}
	\label{alg-solutions at zero}
	\begin{algorithmic}[1]
		\REQUIRE A first-order AODE $F(y,y')=0$, where $F\in \C[y,p]$ is square-free with no factor in $\C[y]$ or $\C[p]$.
		\ENSURE A set consisting of all determined solution truncations of $F(y,y')=0$ expanded around zero.
		\STATE If $(\infty,\infty) \in \cc$, then perform the transformation $\tilde{y}=1/y$ and apply the algorithm to the numerator of $F(1/y,-p/y^2)$ and $\qq=(0,0)$.
		\STATE Compute the set of critical curve points $\mathcal{B}(F)$.
		\STATE For every point $(y_0,p_0) \in \cc \setminus \mathcal{B}(F)$ a determined solution truncation is $y_0+p_0x$.
		\STATE For every critical curve point $\qq=(y_0,p_0) \in \mathcal{B}(F)$ with $y_0 \in \C$ we compute the finite set $\RT$, where $N=2(\deg_p(F)-1) \deg_y(F)+1$.
		\STATE If $p_0=0$, then add to the output the constant solution $y(x)=y_0$.
		\STATE For every truncation $(\hat{a}(t),\hat{b}(t)) \in \RT$ corresponding to $[(a(t),b(t)] \in \Places$, equation \eqref{eq:nec_eq_order_ramification} can be checked.
		\STATE In the negative case, $[(a(t),b(t)]$ is not a solution place.
		\STATE In the affirmative case compute by the Newton polygon method for differential equations the first $N$ terms of the solutions $s_1(t),\ldots,s_n(t)$ of \eqref{eq:DE_2}, denoted by $\hat{s}_1(t),\ldots,\hat{s}_n(t)$.
		\STATE Then the first $N$ terms of $\hat{a}(\hat{s}_i(x^{1/n}))$ are the determined solution truncations with $\qq$ as initial tuple.
	\end{algorithmic}
\end{algorithm}

In the following theorem we show that the output truncations are indeed determined solution truncations, which also proves correctness of Algorithm PuiseuxSolve.

\begin{theorem} \label{cor: bound}
	Let $F \in \C[y,p]$ be square-free with no factor in $\C[y]$ or $\C[p]$.
	Then the set of truncated solutions obtained by the Algorithm PuiseuxSolve with $\qq$ as initial tuple, denoted by $\ST$, and $\Sol$ are in one-to-one correspondence.
\end{theorem}
\begin{proof}
	Let $\qq \in \cc$.
	From \cite{Duval1989} it follows that $\# \RT = \# \Places$.
	By Proposition \ref{Theorem:Characterization Sol Param}, $\# \Sol= \sum_{i=1}^{\# \Places} n_i$ where every summand $n_i$ is equal to the ramification order of the corresponding place (or $0$, if \eqref{eq:nec_eq_order_ramification} is not fulfilled).
	It remains to prove that $\# \ST = \sum_{i=1}^{\# \RT} n_i$, or in other words, that all output elements generated by Algorithm PuiseuxSolve are distinct.
	
	Let $\hat{A}_i=(\hat{a}_i(t),\hat{b}_i(t)) \in \RT$ with $\hat{a}_i(t)=y_0+t^{k_i}$, $\hat{b}_i(t)=\sum_{j \geq 0}^{H} b_{i,j}t^{r_i+j}$ and $n_i=k_i-r_i>0$ for $i=1,2$.
	If $k_1 \neq k_2$ or $r_1 \neq r_2$, the statement holds. 
	So let us assume that $k=k_1=k_2$ and $r=r_1=r_2$ and set $n=k-r$.
	
	Let $\hat{A}_1 \neq \hat{A}_2$. 
	If the quotient $b_{1,0}/b_{2,0}$ is equal to $\lambda \in \C$ with $\lambda^{k}=1$, then let us choose the classical Puiseux parametrization $\hat{A}_2(\lambda t)$ instead of $\hat{A}_2$. 
	Then $b_{1,0}=b_{2,0}$. 
	By definition of $\RT$, $\hat{b}_1(t) \neq \hat{b}_2(t)$. 
	Let us choose $m \in \mathbb{Z}_{\geq 0}$ as the first index such that $b_{1,m} \neq b_{2,m}$. 
	Note that if $m=0$, then $b_{1,0}^k \neq b_{2,0}^k$.
	
	Let $\hat{s}_i(t)=\sum_{j=1}^{H} \sigma_{i,j}\,t^j$ be the truncated solutions of~\eqref{eq:DE_2} corresponding to $\hat{A}_i(t)$.
	First, assume that $m=0$. 
	Then $$\sigma_{1,1}=\sqrt[n]{\frac{n\,b_{1,0}}{k}} \neq \sigma_{2,1}=\sqrt[n]{\frac{n\,b_{2,0}}{k}}$$ and also $\sigma_{1,1}^k \neq \sigma_{2,1}^k$. 
	The coefficient of $t^k$ in $\hat{a}_i(\hat{s}_i(t))$ is equal to $\sigma_{i,1}^k$ and thus, the outputs $\hat{a}_i(\hat{s}_i(x^{1/n}))$ are distinct already in the first coefficient.
	
	Now let us consider $m>0$. 
	Then $\sigma_{1,1}^n=\sigma_{2,1}^n$ and without loss of generality we can choose $\sigma_{1,1}=\sigma_{2,1}$.
	By item (b) in Lemma \ref{le:existence_and_convegence_of_reparametriztion}, and the fact that $\hat{a}_1(t)=\hat{a}_2(t)$, the coefficients $\sigma_{1,1},\ldots,\sigma_{1,m}$ coincide with $\sigma_{2,1},\ldots,\sigma_{2,m}$. 
	In order to show that $\sigma_{1,m+1} \neq \sigma_{2,m+1}$ we consider the coefficient of $t^{m}$ in \eqref{eq:DE_BriotBouquet} by setting $z_i(t)=\sum_{i=2}^H \sigma_i\,t^{i-1}$ and $\sigma=\sigma_{i,1}$. 
	On the left hand side we obtain the coefficient of $t^{m}$ in $k\,(\sigma_{i,1}+z_i(t))^{\nu-1}\,t\,z'(t)$, which is equal to $k\,m\,\sigma_{i,1}^{\nu-1}\,\sigma_{i,m+1}$ plus a polynomial expression in $\sigma_{i,1},\ldots,\sigma_{i,m}$. 
	On the right hand side we obtain from the first summand $n\,b_m\,\sigma_{i,1}^m$ and from the second summand $-k\,(\nu-1)\,\sigma_{i,1}^{\nu-1}\,\sigma_{i,m+1}$ plus a polynomial expression in $\sigma_{i,1},\ldots,\sigma_{i,m},b_0,\ldots,b_{m-1}$.
	We can uniquely solve the corresponding equation for $\sigma_{i,m+1}$, where one summand is equal to $n\,b_{i,m}\,\sigma_{i,1}^m$ and the other summands do not depend on $b_{i,m}$; compare with \eqref{eq:Brior-Bouquer_recursive_equation}. 
	Since $b_{1,m} \neq b_{2,m}$, it follows that $n\,b_{1,m}\,\sigma_{1,1}^m \neq n\,b_{2,m}\,\sigma_{2,1}^m$ and therefore, $\sigma_{1,m+1} \neq \sigma_{2,m+1}$. 
	The coefficient of $t^{k+m}$ in $\hat{a}_i(\hat{s}_i(t))$ is equal to $k\,\sigma_{i,1}^{k-1}\, \sigma_{i,m+1}$ plus a polynomial expression in $\sigma_{i,1},\ldots,\sigma_{i,m}$.
	Thus, the outputs $\hat{a}_i(\hat{s}_i(x^{1/n}))$ are distinct.
\end{proof}

\begin{example}[Example 2 in \cite{FalkensteinerSendra_2018}]
	Let us consider $$F=((y'-1)^2+y^2)^3-4(y'-1)^2y^2=0.$$
	The generic solution is given by $y(x;y_0)=y_0+p_0x+\mathcal{O}(x^2) \in \C[[x]]$ with $p_0 \in \C$ such that $F(y_0,p_0)=0$ and $\frac{\partial F}{\partial p}(y_0,p_0) \neq 0$. 
	The critical set is $\mathcal{B}= \{(0,1),(\alpha,0), (\frac{4\beta}{9},\gamma),(\infty,\infty)\}$ where $\alpha^6+3\alpha^4-\alpha^2+1=0$, $\beta^2=3,$ and $27\gamma^2-54\gamma +19=0$.
	Observe that, since the leading coefficient of $F$ w.r.t. $y$ is $1$, there is no curve point of the form $(y_0,\infty)$ with $y_0 \in \C$.
	
	We now analyze the critical curve points. Let $\mathbf{c}_{\alpha}=(\alpha,0)$ where $\alpha^6+3\alpha^4-\alpha^2+1=0$.
	We get the place
	$$\left(\alpha+t,\frac{\alpha}{19}\,(11\alpha^4+36\alpha^2+4)\,t+ \mathcal{O}(t^2) \right),$$
	which does not provide any solution (see equation \eqref{eq:nec_eq_order_ramification}).
	Thus, the constant $\alpha$ is the only solution with the initial tuple $\mathbf{c}_{\alpha}$.
	
	Let $\mathbf{c}_1=(0,1)$.
	The truncated classical Puiseux parametrizations at $\mathbf{c}_1$ are
	\[ \begin{array}{ll}
	(a_1(t),b_1(t)) = (t^2,1+\sqrt{2} t-\frac{3t^3}{4 \sqrt{2}}-\frac{15t^5}{64 \sqrt{2}}+\mathcal{O}(t^6)), \\[0.1cm]
	(a_2(t),b_2(t)) = (t^2,1-\sqrt{2}it-\frac{3it^3}{4 \sqrt{2}}+\frac{15 it^5}{64 \sqrt{2}}+\mathcal{O}(t^6)), \\[0.1cm]
	(a_3(t),b_3(t)) = (t,1+\frac{it^2}{2}+\frac{3it^4}{8}+\mathcal{O}(t^6)), \\[0.1cm]
	(a_4(t),b_4(t)) = (t,1-\frac{it^2}{2}-\frac{3it^4}{8}+\mathcal{O}(t^6)).
	\end{array} \]
	So we have $n=2$ for $(a_1(t),b_1(t))$ and $(a_2(t),b_2(t))$ and $n=1$ for $(a_3(t),b_3(t))$ and $(a_4(t),b_4(t))$.
	Then equation \eqref{eq:DE_2} corresponding to $(a_1(t),b_1(t))$ is $$s(t)\,s'(t)=t\,\left(1+\sqrt{2}s(t)-\frac{3s(t)^3}{4 \sqrt{2}}-\frac{15s(t)^5}{64 \sqrt{2}}\right).$$
	We obtain the solutions
	\[ \begin{array}{ll}
	s_1(t)=t+\frac{\sqrt{2}t^2}{3}+\frac{t^3}{18}-\frac{89t^4}{540\sqrt{2}}+\mathcal{O}(t^5), \\[0.1cm] s_2(t)=-t+\frac{\sqrt{2}t^2}{3}-\frac{t^3}{18}-\frac{89t^4}{540\sqrt{2}}+\mathcal{O}(t^5).
	\end{array} \]
	Therefore, $a_1(s_1(x^{1/2}))$ and $a_1(s_2(x^{1/2}))$ are determined solution truncations of $F(y,y')=0$.
	
	Similarly we can find two determined solution truncations coming from $(a_2(t),b_2(t))$ and one for each $(a_3(t),b_3(t))$ and $(a_4(t),b_4(t))$.
	We note that the solutions corresponding to $(a_3(t),b_3(t))$ and $(a_4(t),b_4(t))$ are formal power series and already detected in \cite{FalkensteinerSendra_2018}.
	Thus,
	\[
	\mathcal{U}_{\mathbf{c}_1}=\left\{
	\begin{array}{ll}
	a_1(s_1(x^{1/2}))=x+\frac{2\,\sqrt{2}\,x^{3/2}}{3}+\frac{x^2}{3}+\mathcal{O}(x^{5/2}), \\[0.1cm]
	a_1(s_2(x^{1/2}))=x-\frac{2\sqrt{2}x^{3/2}}{3}+\frac{x^2}{3}+\mathcal{O}(x^{5/2}), \\[0.1cm]
	a_2(\tilde{s}_1(x^{1/2}))=x+\frac{2\sqrt{2}ix^{3/2}}{3} -\frac{x^2}{3}+\mathcal{O}(x^{5/2}), \\[0.1cm]
	a_2(\tilde{s}_2(x^{1/2}))=x-\frac{2\sqrt{2}ix^{3/2}}{3} -\frac{x^2}{3}+\mathcal{O}(x^{5/2}), \\[0.1cm]
	a_3(s(x))=x+\frac{x^3}{6} +\frac{17x^5}{240}+\mathcal{O}(x^6), \\[0.1cm]
	a_4(s(x))=x-\frac{x^3}{6} +\frac{17x^5}{240}+\mathcal{O}(x^6)
	\end{array}
	\right\} \]
	is the set of all determined solution truncations with $\mathbf{c}_1$ as initial tuple.
	
	Let $\mathbf{c}_{\beta,\gamma}=\left(\frac{4\beta}{9},\gamma\right)$, where $\beta^2=3,$ and $ 27\gamma^2-54\gamma +19=0$.
	We get the place $$\left(\frac{4\beta}{9}+t^2,\gamma+\frac{\sqrt{\beta}\,i}{\sqrt{3}}t+\mathcal{O}(t^2)\right).$$
	Thus, \eqref{eq:nec_eq_order_ramification} is fulfilled with $n=2$.
	Similarly as before, we obtain at $\mathbf{c}_{\beta,\gamma}$ the set of solutions
	\[
	\mathcal{U}_{\mathbf{c}_{\beta,\gamma}}=\left\{
	\begin{array}{ll}
	\frac{4 \beta}{9}+\gamma x+\frac{2\sqrt{-\gamma\,\beta}}{3 \sqrt{3}} x^{3/2} + \left(\frac{5 \gamma}{32}-\frac{143}{864} \right) \beta x^2 +\mathcal{O}(x^{5/2}), \\[0.1cm]
	\frac{4 \beta}{9}+\gamma x-\frac{2\sqrt{-\gamma\,\beta}}{3 \sqrt{3}} x^{3/2} + \left(\frac{5 \gamma}{32}-\frac{143}{864} \right) \beta x^2 +\mathcal{O}(x^{5/2})
	\end{array}
	\right\}. \]
	
	Let us analyze $\mathbf{c}_{\infty}=(\infty,\infty)$.
	The numerator of $F(1/y,-y'/y^2)$ is equal to
	\begin{align*}
	G=&y^{12}+(6y'-1)y^{10}+(15y'^2+4y'+3)y^8+(20y'^3+14y'^2+6y'+1)y^6+ \\
	&(15y'^4+12y'^3+3y'^2)y^4+(6y'^5+3y'^4)y^2+y'^6.
	\end{align*}
	The places at the origin of $\mathscr{C}(G)$ are given by $$(t^3,\pm i\,t^3 + \mathcal{O}(t^4)),$$ which do not define a solution place.
	
	Now the set $\{y(x;y_0)\} \cup~ \{\alpha\} \cup~ \mathcal{U}_{\mathbf{c}_1} \cup~ \mathcal{U}_{\mathbf{c}_{\beta,\gamma}}$ describes all formal Puiseux series solutions of $F=0$.
\end{example}

\subsection{Solutions expanded at infinity} \label{subsec-infinity}
In this subsection we describe the formal Puiseux series solutions of \eqref{eq-main} expanded around infinity, or equivalently up to the sign, formal Puiseux series solutions of \eqref{eq-infinity} with $h=2$ expanded around zero. 
The different sign in the second component requires to use, in all reasonings and results, a change in the sign of the second component of the formal parametrizations as well. 
This implies, in the same notation as in equation \eqref{eq:DE_2}, the slightly different associated differential equation for solutions expanded around infinity, namely
\begin{equation} \label{eq-assocInfinity}
a'(s(t))\,s'(t)=-m\,t^{-m-1}\,b(s(t)).
\end{equation}

As we have already remarked, condition \eqref{eq:nec_eq_order_ramification} is only a necessary and not a sufficient condition. 
That there are some cases which fulfill \eqref{eq:nec_eq_order_ramification} for some $n \in \mathbb{Z}_{>0}$ but do not lead to solutions is shown in the following example.
\begin{example}
	We are looking for formal Puiseux series solutions expanded around infinity of $F(y,y')=(1+y)\,y'+y^2.$ 
	Instead we consider $$F(y,-x^2y')=-x^2\,(1+y)\,y'+y^2=0.$$
	By the Newton polygon method for differential equations we directly see that there is no formal Puiseux series solution with $y(0)=0$ except the constant zero.
	On the other hand, $(t,-t^2+\mathcal{O}(t^3))$ is a formal parametrization of $\cc$ fulfilling \eqref{eq:nec_eq_order_ramification} with $n=1$.
\end{example}

Nevertheless, it can still be checked whether there exists a solution fulfilling the necessary condition \eqref{eq:nec_eq_order_ramification} or not and therefore, we can algorithmically compute all solutions as in the previous subsection.

Similar to Section \ref{subsec-zero} let $\RT \subseteq \C[t][t^{-1}]$ denote the set of truncations of non-equivalent classical Puiseux parametrizations where the first $N$ terms are computed.

\begin{algorithm}[H]
	\caption{PuiseuxSolveInfinity}
	\label{alg-solutions at infinity}
	\begin{algorithmic}[1]
		\REQUIRE A first-order AODE $F(y,y')=0$, where $F\in \C[y,p]$ is square-free with no factor in $\C[y]$ or $\C[p]$.
		\ENSURE A set consisting of all solution truncations of $F(y,y')=0$ expanded around infinity.
		\STATE Compute the algebraic set $\mathbb{V}(F(y,0))$.
		\STATE For every $y_0 \in \mathbb{V}(F(y,0))$ compute the finite set $\RT$, where $N=\max(2(\deg_p(F)-1) \deg_y(F)+1,\deg_y(F)+1)$.
		\STATE Add to the output the constant solutions $y(x)=y_0$.
		\STATE For every truncation $(\hat{a}(t),\hat{b}(t)) \in \RT$ corresponding to $[(a(t),b(t)] \in \Places$, check equation \eqref{eq:nec_eq_order_ramification}.
		\STATE In the negative case, $[(a(t),b(t)]$ is not a solution place.
		\STATE In the affirmative case check by the Newton polygon method for differential equations whether~\eqref{eq-assocInfinity} is solvable. 
		Note that the critical term with index $n$ is already covered by the first $N$ terms, since $n \leq \deg_y(F)+1 \leq N$.
		\STATE In the affirmative case compute the first $N$ terms of the solutions $s_1(t),\ldots,s_n(t)$ denoted by $\hat{s}_1(t),\ldots,\hat{s}_n(t)$, which contain a free parameter.
		\STATE The first $N$ terms of $\hat{a}(\hat{s}_i(x^{-1/n}))$ are solution truncations with $y_0$ as initial value.
	\end{algorithmic}
\end{algorithm}

Let $y_0 \in \C$ be such that $\qq=(y_0,0) \in \cc$ and let us again denote the formal Puiseux series solutions expanded around infinity with $\qq$ as initial tuple by $\Sol$ and the output of Algorithm PuiseuxSolveInfinity by $\ST$.

Since $\mathbb{V}(F(y,0))$ is a finite set and termination of the Newton polygon method for computing formal parametrizations and of the Newton polygon method for computing the reparametrizations is ensured, also termination of Algorithm \ref{alg-solutions at infinity} follows.
Correctness of Algorithm \ref{alg-solutions at infinity} follows from Section \ref{sec-PSP} and the following corollary.

\begin{corollary} \label{cor: bound infinity}
	Let $F \in \C[y,p]$ be square-free with no factor in $\C[y]$ or $\C[p]$.
	Then every solution truncation $\tilde{y}(x) \in \ST$ can be extended to $y(x) \in \Sol$ and conversely, for every $y(x) \in \Sol$ there exists a truncation $\tilde{y}(x) \in \ST$.
\end{corollary}
\begin{proof}
	In Algorithm \ref{alg-solutions at infinity} all places fulfilling the necessary condition \eqref{eq:nec_eq_order_ramification} are treated. 
	By Lemma \ref{le:existence_and_convegence_of_reparametriztion}, all solutions of~\eqref{eq-assocInfinity} are found and by Proposition \ref{Theorem:Characterization Sol Param} the statement holds.
\end{proof}

In Theorem \ref{cor: bound} we were able to additionally show that the corresponding output truncations $\tilde{y}_1, \tilde{y}_2 \in \RT$ coming from different places or different reparametrizations do not coincide.
However, in Algorithm PuiseuxSolveInfinity we cannot guarantee this.
The problem for adapting the proof of Theorem \ref{cor: bound} lies in the free parameter of the reparametrizations.

\begin{example} \label{ex-infinity}
	Let us consider $$F(y,y')=y'+y^2=0$$ and its formal Puiseux series solutions expanded around infinity.
	We obtain $\mathbb{V}(F(y,0))=\{0\}$.
	For $\qq=(0,0)$ compute the formal parametrization $(a,b)=(t,-t^2)$, which fulfills \eqref{eq:nec_eq_order_ramification} with $n=1$.
	Equation \eqref{eq-assocInfinity} simplifies to $s'(t)=t^{-2}\,s(t)^2$ having the solutions $s(t)=\frac{t}{1-c\,t}=t+c\,t^2+c^2\,t^3+\mathcal{O}(t^4)$ for an arbitrary constant $c \in \C$.
	Hence, $$a(s(x^{-1}))=\frac{1}{x-c}=\frac{1}{x}+\frac{c}{x^2}+\frac{c^2}{x^3}+\mathcal{O}(x^{-4})$$ describes all formal Puiseux series solutions expanded around infinity.
\end{example}


\appendix

\bibliographystyle{alpha}

\end{document}